\documentclass[12pt]{amsart}

\usepackage[a4paper,centering,margin=2.5cm]{geometry}
\usepackage{color}
\usepackage{amssymb}
\usepackage{graphicx,psfrag}
\title{Detecting minors in matroids through triangles}
\dedicatory{Dedicated to the memory of Michel Las Vergnas}
\author[Boris Albar]{Boris Albar}
\address{Institut de Math\'{e}matiques et de Mod\'{e}lisation de Montpellier, Universit\'{e} Montpellier 2, Case Courrier 051, Place Eug\`{e}ne Bataillon, 34095 Montpellier Cedex 05, France and
LIRMM, Universit\'{e} Montpellier 2, 161 rue Ada, 34095 Cedex 05, France}
\email{Boris.Albar@lirmm.fr}
\author[Daniel Gon\c calves]{Daniel Gon\c calves}
\address{LIRMM, Universit\'{e} Montpellier 2, 161 rue Ada, 34095 Cedex 05, France}
\email{Daniel.Goncalves@lirmm.fr}
\author[Jorge L. Ram\'irez Alfons\'in]{Jorge L. Ram\'irez Alfons\'in}\thanks{This work was supported by the ANR grants TEOMATRO ANR-10-BLAN 0207 and EGOS 12 JS02 002 01.}
\address{Institut de Math\'{e}matiques et de Mod\'{e}lisation de Montpellier, Universit\'{e} Montpellier 2, Case Courrier 051, Place Eug\`{e}ne Bataillon, 34095 Montpellier Cedex 05, France}
\email{jramirez@math.univ-montp2.fr}

\date{\today}

\theoremstyle{plain}
\newtheorem{theorem}{Theorem}
\newtheorem{lemma}{Lemma}

\theoremstyle{definition}

\newtheorem{question}{Question}

\def\boxit{$\sqcap\kern-8pt\sqcup$}

\begin{document}

\maketitle

\begin{abstract} In this note we investigate some matroid minor structure results.
In particular, we present sufficient conditions, in terms of {\em triangles}, for a matroid to
have either $U_{2,4}$ or $F_7$ or $M(K_5)$ as a minor.
\end{abstract}

\section{Introduction}

In \cite{Mader} Mader proved that, for each $3\le r\le 7$, if a graph $G$ on $n$
vertices has no $K_r$ minor then it has at most $n(r-2)-{{r-1}\choose 2}$ edges.
The latter was used by Nevo \cite{Nevo} to show that, for $3\le r\le 5$, if each
edge of $G$ belongs to at least $r-2$ triangles then $G$ has a $K_r$ minor. The latter also
holds when $r = 6,7$ (see \cite{AG}).
In the same flavour, we investigate similar conditions for a matroid in order to
have certain minors.
For general background in matroid theory we refer the reader to \cite{Oxley,Welsh}.
A {\em triangle} in a matroid is just a circuit of cardinality three. Our
main result is the following.

\begin{theorem}\label{th:main} Let $M$ be a simple matroid. If every element of $M$ belongs to at
least three triangles then $M$ has $U_{2,4}$, $F_7$ or $M(K_5)$ as a minor.
\end{theorem}

We notice that excluding $U_{2,4}$ as a submatroid (instead of as a minor) would not be sufficient
as shown by the matroid $AG(2,3)$.
Note that $AG(2,3)$ doesn't contain $M(K_4)$ as a minor. Hence, it doesn't contain $M(K_5)$ nor
$F_7$ as a minor. Moreover each element of the matroid $AG(2,3)$ belongs to $4$ triangles but it
has no $U_{2,4}$ submatroid. In the same way, graphic matroids (that are $U_{2,4}$ and $F_7$-minor free)
imply that $M(K_5)$ cannot be simply excluded as a submatroid. We do not know if excluding $F_7$ as a submatroid in Theorem
\ref{th:main} would be sufficient.

%

A natural question is whether similar triangle conditions can be used to determine
if a matroid admits $M(K_4)$ as a minor. More precisely,
\smallskip

\begin{quote} {\em is it true that if every element of a matroid $M$ of rank $r\ge 3$
belongs to at least two triangles then $M$ contains $M(K_4)$ as a minor ?}
\end{quote}
\smallskip

The answer to this question is yes if $M$ is regular (we discuss this at the end of this section see \eqref{gen-nevo}).
Moreover, the answer is still yes if $M$ is binary since the class of
binary matroids without a $M(K_4)$-minor is the class of series-parallel graphs
(a result due to Brylawski \cite{Bra}). Unfortunately, the answer is no in general,
for instance, the reader may take the matroid $P_7$, illustrated in Figure~\ref{f1},
as a counterexample.

\begin{figure}[ht]
\includegraphics[scale=.5]{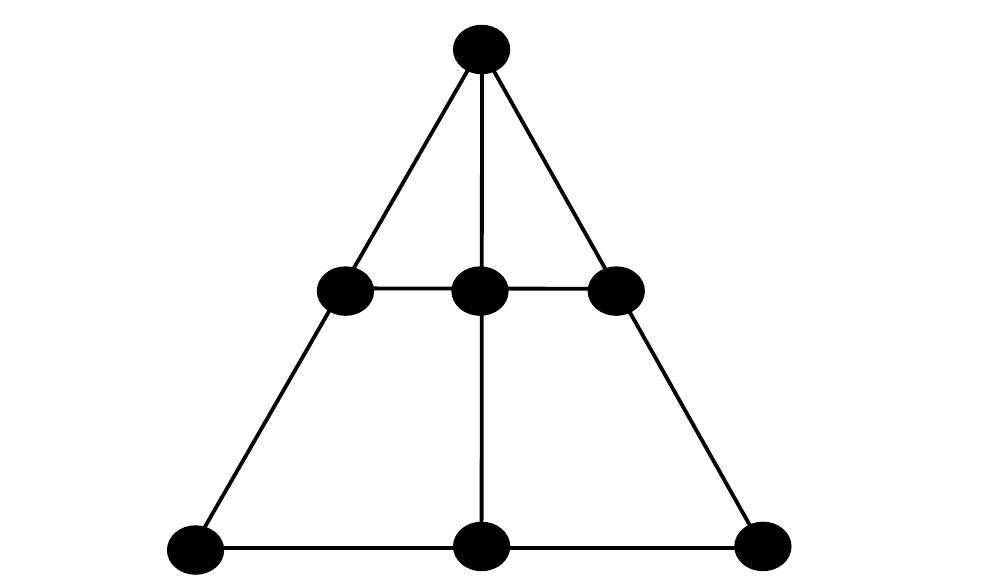}
\caption{Matroid $P_7$.}\label{f1}
\end{figure}

In the case of ternary matroids, we prove the following.
\begin{theorem}\label{th:ternary}
Every simple ternary matroid $M$ such that every element belongs to at least $3$ triangles
contains a $P_7$ or a $M(K_4)$ minor or contains the matroid $U_{2,4}$ as a submatroid.
\end{theorem}

The above yielded us to consider the following

\begin{question} Does there exist two finite lists $\mathcal L$ and $\mathcal S$ of matroids such that
(a) for each $M\in \mathcal L \cup \mathcal S$, each element $e\in M$ belongs to at least $t$ triangles, and
(b) for any matroid $M$ such that each of its elements belong to $t$ triangles, $M$ contains one of the matroids in $\mathcal L$ as a minor
or $M$ contains one of the matroids in $\mathcal S$ as a submatroid ?
\end{question}

It is easy to see that $U_{2,k}$ will belong to either $\mathcal L$ or $\mathcal S$ since it is a matroid
with smallest rank such that each element belong to $t$ triangles for $k$ big enough (depending on $t$).
Moreover, the matroid $M(K_{t+2})$ will also
always be contained in one of these lists since each edge of $K_{t+2}$ belongs to exactly $t$ triangles.
We finally mention the following generalization of Nevo's result:

\begin{align}\label{gen-nevo}
&\text{If every element of a simple regular matroid $M$ belongs to at
least $r-2$}\\
&\text{triangles, with $3\le r\le 7$, then $M$ has $M(K_r)$ as a minor.}\nonumber
\end{align}

Although this can be proved by applying essentially the same methods as those used in the proof of Theorem \ref{th:main},
we rather prefer to avoid to do this here since the arguments need a more detailed treatment (specially when $r=6,7$). 

\section{Proof of Theorem~\ref{th:main}}

We start by recalling some basic definitions and results needed throughout the paper.
We shall denote by $\mathcal{C}(M)$ the set of circuits of a matroid $M$.
Let $k$ be a positive integer. Then, for a matroid $M$, a partition $(X,Y)$ of
$E(M)$ is a {\em $k$-separation} if $\min\{|X|,|Y|\}\ge k$ and $r(X)+r(Y)-r(M)\le k-1$.
$(X,Y)$ is called an {\em exact $k$-separator} if $r(X)+r(Y)-r(M)=k-1$. $M$ is called
{\em $k$-separated} if it has a $k$-separation. If $M$ is $k$-separated for some $k$,
then the connectivity $\lambda(M)$ of $M$ is $\min\{j : M \text{ is } j\text{-separated }\}$;
otherwise we take $\lambda(M)$ to be $\infty$. We say that a matroid is {\em $k$-connected}
if $\lambda(M) \geq k$.
\medskip

Let $M_1$ and $M_2$ be two matroid with non-empty ground set $E_1$ and $E_2$
respectively. Let
\[\mathcal{C'} = \mathcal{C}(M_1 \setminus (E_1 \cap E_2)) \cup \mathcal{C}(M_2 \setminus (E_1 \cap E_2)) \cup
\{ C_1 \Delta C_2 \: : \: C_i \in \mathcal{C}(M_i) \textrm{ for } i = 1,2\}.\]
We denote by $\mathcal{C}$ the set of minimal elements (by inclusion) of $\mathcal{C'}$.
\begin{itemize}
\item If $|E_1 \cap E_2| = 0$, then $\mathcal{C}$ is the set of circuits of a matroid with support $E_1 \Delta E_2$ called
the {\em $1$-sum} or {\em direct sum} of $M_1$ and $M_2$ and denoted by $M_1 \oplus_1 M_2$.
\item If $|E_1 \cap E_2| = 1$, $|E_1|,|E_2| \geq 3$ and $E_1 \cap E_2$ is not a loop or a coloop of either $M_1$
or $M_2$, then $\mathcal{C}$ is the set of circuits of a matroid with support $E_1 \Delta E_2$ called the {\em $2$-sum}
of $M_1$ and $M_2$ and denoted by $M_1 \oplus_2 M_2$.
\item If $M_1$ and $M_2$ are binary matroids with $|E_1 \cap E_2| = 3$, $|E_1|,|E_2| \geq 7$,
such that $E_1 \cap E_2$ is a circuit of both $M_1$ and $M_2$ and such that $E_1 \cap E_2$
contains no cocircuit of either $M_1$ or $M_2$, then $\mathcal{C}$ is the set of circuits of a binary matroid with support
$E_1 \Delta E_2$ called the {\em $3$-sum} of $M_1$ and $M_2$ and denoted by $M_1 \oplus_3 M_2$.
\end{itemize}

The following structural result is a consequence of Seymour's results in
\cite{Seymour} (see also \cite[Corollary 11.2.6]{Oxley}):

\begin{align}\label{sey}
& \text{\cite[(6.5)]{Seymour} Every binary matroid with no $F_7$ minor can be obtained by a sequence}\\
& \text{$1$- and $2$-sums of regular matroids and copies of $F^*_7$}.\nonumber
\end{align}

The following results, in relation with $k$-separations, are also due to Seymour \cite{Seymour1}.

\begin{align}\label{sey1}
& \text{\cite[(2.1)]{Seymour1} If $(X,Y)$ is a $1$-separator of $M$ then $M$ is the $1$-sum of $M|_X$ and $M|_Y$;}\\
& \text{and conversely, if $M$ is the $1$-sum of $M_1$ and $M_2$ then $(E(M_1),E(M_2))$ is a}\nonumber \\
&\text{$1$-separation of $M$, and $M_1, M_2$ are isomorphic to proper minors of $M$.}\nonumber
\end{align}

\begin{align}\label{sey2}
& \text{\cite[(2.6)]{Seymour1} If $(X,Y)$ is an exact 2-separator of $M$ then there are matroids $M_1$}\\
& \text{$M_2$ on $X\cup\{z\}, Y\cup\{z\}$ respectively (where $z$ is a new element) such that}\nonumber \\
& \text{$M$ is the 2-sum of $M_1$ and $M_2$. Conversely, if $M$ is the 2-sum of $M_1$ and $M_2$}\nonumber\\
& \text{then $(E(M_1)-E(M_2),E(M_2)-E(M_1))$ is an exact 2-separation of $M$,}\nonumber \\
& \text{and $M_1, M_2$ are isomorphic to proper minors of $M$.}\nonumber
\end{align}

\begin{align}\label{sey4}
&\text{\cite[(4.1)]{Seymour1} If $M$ is a $3$-connected binary matroid and is the 3-sum of two}\\
&\text{matroids $M_1$ and $M_2$, then $M_1$ and $M_2$ are isomorphic to proper minors of $M$.}\nonumber
\end{align}

\begin{align}\label{sey5}
&\text{\cite[(2.10)]{Seymour1} A $2$-connected matroid $M$ is not $3$-connected if and only if }\\
&\text{$M = M_1 \oplus_{2} M_2$ for some matroids $M_1$ and $M_2$, each of which is isomorphic}\nonumber \\
&\text{ to a proper minor of $M$.}\nonumber
\end{align}

We shall use \eqref{sey}-\eqref{sey5} and the following three lemmas to prove
our main theorem. 
We will denote by $si(M)$ the matroid obtained from $M$ by deleting all its loops and by identifying parallel elements.

\begin{lemma}\label{lem:elemsum}
Let $M_1$ and $M_2$ be two matroids with ground sets $E_1$ and $E_2$ respectively such that
$M = M_1 \oplus_k M_2$, $1\le k\le 3$ and such that $M$ is a simple matroid. Moreover, we suppose that $M$ is binary when  $k=3$.
Let $e \in E_1 \setminus E_2$ such that $\{e,x\}\in\mathcal{I}(M_1)$ for any element 
$x\in E_1 \cap E_2$ and suppose that $e$ belongs to $t$ triangles of $M$. Then, $e$ belongs to at least $t$ triangles of $si(M_1)$.
\end{lemma}

\begin{proof}
Let $e \in E_1 \setminus E_2$ such that $\{e,x\}\in\mathcal{I}(M_1)$ for any element 
$x\in E_1 \cap E_2$ and suppose that $e$ belongs to $t$ triangles of $M$. We shall show that
$e$ belongs to at least $t$ triangles of $si(M_1)$.
\medskip

Let $T = \{e,f,g\}$ be one of the $t$ triangles of $M$ containing $e$ and note that $e,g,f \not\in E_1 \cap E_2$. 
By definition of the $k$-sum, either $T$ is a circuit of
$\mathcal{C}(M_1)$ and we are done, or $T$ can be written as $C_1 \Delta C_2$
where $C_i$ is a circuit of $M_i, i=1,2$. Since $M$ is simple and $E_1 \cap E_2$
contains no loop (by definition of $k$-sum) then neither $M_1$ nor $M_2$ contain a 
loop, and thus $|C_1|,|C_2| \geq 2$.
\medskip

If $|C_1| = 2$, say $C_1 = \{e,x\}$, then $x \in E_1 \cap E_2$ (otherwise $e$ and $x$ would be parallel elements in $M$, contradicting the simplicity of $M$).
So, $e$ is parallel to an element $x$ with $x \in E_1 \cap E_2$  contradicting the hypothesis of the lemma. We have then that $|C_1|\ge 3$
\medskip

If $|C_2| = 2$, say $C_2 = \{g,x\}$, then $x \in E_1 \cap E_2$ (otherwise $g$ and $x$ would be parallel elements in $M$, contradicting the simplicity of $M$).
Since $f \in T = C_1 \Delta C_2$ then $f \in E_1$ and since $x\in E_1$ is parallel to $g$ then $\{e,f,x\}$ is a triangle of $M_1$.
\medskip

Let us suppose now that $|C_1|,|C_2| \geq 3$. Since $|C_1 \Delta C_2|=|T|=3$ then $|C_1 \cap C_2| \geq 2$.
So we are in the case where $k = 3$ and thus we can suppose that $M$ is binary.
Moreover since $E_1 \cap E_2$ is a circuit of both $M_1$ and $M_2$,
then $C_1$ and $C_2$ contain at most two elements of $E_1 \cap E_2$ or they are equal
to $E_1 \cap E_2$. In the latter, we have that $e \in E_1 \cap E_2$ which is a contradiction since 
$e \in E_1 \setminus E_2$.  We thus suppose that we are in the former. Hence $|C_1|+|C_2| = 7$
and we can write $C_1 \cap C_2 = \{x,y\}$. 
Therefore one of $|C_1|$ or $|C_2|$ has cardinality at least $4$.
\medskip

We shall use a result due to Fournier \cite{Fournier} stating that a matroid
$M$ is binary if and only if whenever $C_1$ and $C_2$ are distinct
circuits and $\{p,q\}$ are elements of $C_1 \cap C_2$, then
there is a circuit in $M$ contained in $C_1 \cup C_2 \setminus \{p,q\}$.
\medskip

We have two cases.

Case a) $|C_2| = 4$ and $|C_1| = 3$. We write $C_1 = \{e,x,y\}$.
By applying Fournier's result to circuits $E_1 \cap E_2 = \{x,y,z\}$ and
$C_1 = \{e,x,y\}$ we obtain that $\{e,z\}$ contains a circuit and since
by hypothesis neither $e$ nor $z$ is a loop, then $e$ and $z$ are parallel
elements, contradicting the hypothesis because $z \in E_1 \cap E_2$.
\medskip

Case b) $|C_1| = 4$ and $|C_2| = 3$. We write $C_2 = \{x,y,g\}$. By Fournier's
result applied to circuits $\{x,y,z\}$ and $C_2$, we deduce that $g$ and $z$ are parallel
elements. Thus $(T \setminus g) \cup \{z\}$ is a triangle of $si(M_1)$ and is
not a triangle of $M$.

\medskip
It remain to check that two different triangles of $M$ containing $e$ induce, by the previous construction,
two different triangles in $si(M_1)$. Let $T$ and $T'$ be two different triangles of $M$ containing $e$
that are not triangles of $M_1$. Note that $T$ and $T'$ have two elements of $M_1$ because otherwise,
as we have previously seen, $e$ would be parallel to an element of $E_1 \cap E_2$, contradicting the hypothesis.
We denote by $w$ (resp. $w'$) the only element of $T$ (resp. $T'$) that belongs to $M_2$.
By construction the two triangles of $si(M_1)$ obtained from $T$ and $T'$ respectively contain $T \setminus \{w\}$
and $T' \setminus \{w'\}$. If $T \setminus \{w\} \neq T' \setminus \{w'\}$, the resulting triangles of $si(M_1)$
are different. Suppose now that
$T \setminus \{w\} = T' \setminus \{w\}$. In the above construction, the elements $w$ and $w'$
are replaced by elements of $E_1 \cap E_2$ repectively parallel to $w$ and $w'$ respectively. Note that
$w$ and $w'$ cannot be parallel to a common element of $E_1 \cap E_2$ (indeed if $w$ and $w'$ were parallel,
it would contradict the simplicity of $M$). So $w$ and $w'$ are parallel to two
distinct elements of $E_1 \cap E_2$, and thus the triangles $T$ and $T'$ induces two different triangles
in $si(M_1)$.
\end{proof}

\begin{lemma}\label{lem1}
Let $M$ be a simple connected graphic matroid such that each of its elements belongs
to at least three triangles except maybe for one element $e$ or for some elements
of a given triangle $T$ of $M$. If $M$ is not isomorphic to $e$ or $T$,
then $M$ contains $M(K_5)$ as a minor.
\end{lemma}

\begin{proof}
Let $G$ be a graph such that $M=M(G)$. We will prove that $G$ contains a $K_5$ minor.
We will denote by $X$ the set of vertices corresponding to the extremities of the edge $e$
or to the vertices of the triangle $T$ depending on the case. In particular, we have that
$|X| \leq 3$. Since $M(G)$ is simple, then $G$ has at least $4$ vertices, so there exists
$u \in V(G) \setminus X$. Since $M(G)$ is connected then $G$ is connected too and so
$\deg(u) \geq 1$. Moreover, every edge incident to $u$ belongs to at least $3$ triangles,
so the graph induced by $N(u)$ (the set of neighbors of $u$) has minimum degree at least $3$.
Dirac \cite{Dirac} proved that if $G$
is a non-null simple graph with no subgraph contractible to $K_4$, then $G$ has a
vertex of degree $\le 2$. Therefore, by Dirac's result, the graph induced by the
vertices in $N(u)$ contain a $K_4$ minor and so the graph induced by $N(u)$ together
with $u$ contain a $K_5$ minor.
\end{proof}

\begin{lemma}\label{lem2}
Let $M$ be a simple matroid and let $X$ be a set of element of $M$ consisting
of either an element $e$ or of the elements of a given triangle $T$ of $M$.
If each element of $M$ belongs to at least three triangles except for the elements of $X$
and if $M$ is not isomorphic to $M|_{X}$, then $M$ is not a cographic matroid.
\end{lemma}

\begin{proof}
We proceed by contradiction. Suppose that there exists a cographic matroid $M$ contradicting the lemma.
Let $G$ be the graph such that $E(G)$ is the ground set of $M$, and such that the circuits
of $M$ are the edge cuts of $G$. We can suppose that $G$ is connected. Moreover, since $M$ is simple (i.e. it contains no loop no
parallel elements), the graph $G$ has no edge cut of size one or two and thus $G$ is 3-edge connected.
Let us call an edge cut \emph{trivial} if it corresponds to all the edges incident to a given vertex $v$.
Note that an edge that belongs to (at least) three 3-edge cuts of $G$, belongs to at least one non-trivial 3-edge cut.
\smallskip

In the case where $M$ has an element $a$ that does not belong to three
triangles, we denote $v$ one of the endpoints of $a$ in $G$.
Now in the case where $M$ has a triangle $T=\{a,b,c\}$
which elements do not necessarily belong to three triangles, the edge
cut $\{a,b,c\}$ in $G$ is either trivial and then we denote $v$ the degre 3
vertex incident to $a$, $b$ and $c$, or non-trivial and then every edge
of $G$ (including $a$, $b$ and $c$) belongs to a non-trivial 3-edge
cut. For every vertex $v \in V(G)$, the graph $G \setminus \{v\}$
is not a stable set. Indeed, suppose that every edge of $G$ is incident to $v$ then the graph $G$
is isomorphic to a star (with eventually multiples edges and loops on $v$), and so, by a result of Whitney \cite{Whitney} the dual matroid of $M(G)$
(which is isomorphic to $M$) is a graphic matroid associated to the dual graph $G^*$. Thus, since $G$ is a star (with eventually
multiple edges and loops on its center), then $G^*$ is also a star (with eventually multiple edges and loops on its center) of multiples edges.
This contradict the fact that each element of $M$ except at most $3$ belongs to at least $3$ triangles.
which contradicts the simplicity of $M$.

We claim that

\begin{align}\label{claim1}
&\text{there is no 3-edge connected graph $G$, with a vertex $v$, such that every edge} \\
&\text{$e\in E(G\setminus \{v\})$ belongs to some non-trivial 3-edge cut of $G$ and}\nonumber \\
&\text{such that $G \setminus \{v\}$ is not a stable set.}\nonumber
\end{align}

It is clear that the above claim contradicts the existence of $G$ and thus implies the lemma.
We may now prove \eqref{claim1} by contradiction. So let us consider a graph $G$
that is 3-edge connected with a distinguished vertex $v$, and such
that every edge $e\in E(G\setminus \{v\})$ belongs to at least one
non-trivial 3-edge cut of $G$. By hypothesis, the graph $G \setminus \{v\}$
is not a stable set, so there are edges in $G\setminus \{v\}$, $G$ has some non-trivial $3$-edge
cuts.  Let $\{e_1,e_2,e_3\} \subset E(G)$ be a non-trivial 3-edge cut
of $G$, partitioning $V(G)$ into two sets $V_1$ and $V_2$ such that $v
\in V_1$ and such that $|V_2|$ is minimal (see Figure~\ref{f3}). As this edge cut is
non-trivial, there are at least two vertices in $V_2$, and as $G$ is
3-edge connected there is an edge $f_1$ in $G[V_2]$. By hypothesis,
let $\{f_1,f_2,f_3\} \subset E(G)$ be a non-trivial 3-edge cut of $G$,
partitioning $V(G)$ into two sets $X$ and $Y$ such that $v\in
X$. Consider now the refined partition defined by the following sets:
$V_1^X =V_1\cap X$, $V_1^Y =V_1\cap Y$, $V_2^X =V_2\cap X$, and $V_2^Y
=V_2\cap Y$.  Note that as $v\in V_1^X$ and as $f_1$ has both ends in
$V_2$, the sets $V_1^X$ $V_2^X$ and $V_2^Y$ are non-empty.  Note also
that by definition $|V_2| \leq |Y|$, and thus $|V_2^X| \leq
|V_1^Y|$. This implies that the set $V_1^Y$ is also non-empty.

\begin{figure}[ht]
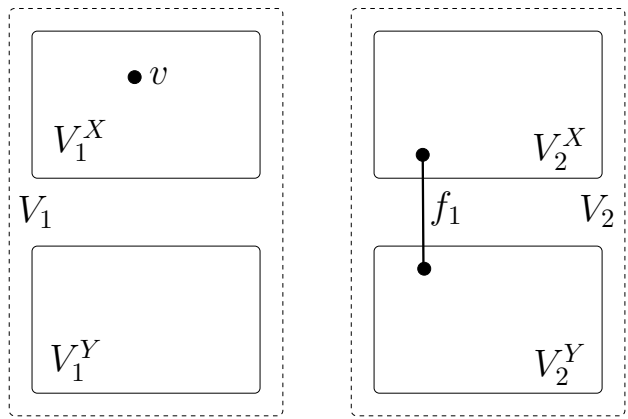
\caption{The $3$-edge connected graph $G$, with edges cuts $\{e_1,e_2,e_3\}$ and $\{f_1,f_2,f_3\}$.}\label{f3}
\end{figure}

By construction, there are at most 6 edges across this partition (if
$\{e_1,e_2,e_3\}$ and $\{f_1,f_2,f_3\}$ are disjoint).  On the other
hand, as $G$ is 3-edge connected each subset of the partition (as they
are non-empty) has at least 3 edges leaving it.  This implies that
there are exactly 6 edges across the partition and that each set has
exactly 3 of them leaving it. Let $0\le k_e\le 3$ be the number of
edges from $\{e_1,e_2,e_3\}$ adjacent to $V_2^X$, and note that
$V_2^Y$ is adjacent to $k'_e =3 - k_e$ of these edges. On the other
hand, there is no edge of $\{f_1,f_2,f_3\}$ going across $V_1$ and
$V_2$, thus the number $k_f$ of edges from this set that are incident
to $V_2^X$ is the same as the number $k'_f$ of edges from this set
that are incident to $V_2^Y$. As $k_e \neq k'_e$ this contradicts the
fact that both $V_2^X$ and $V_2^Y$ are incident to exactly $k_e + k_f
= k_e' + k_f' = 3$ edges. This concludes the proof of the claim.
\end{proof}

We may now prove Theorem \ref{th:main}.
\smallskip

{\em Proof of Theorem \ref{th:main}.} We proceed by contradiction. Let $M$ be a
matroid such that every element belongs to at least three triangles except maybe
for one element $e$ or for some elements of a given triangle $T$ of $M$ and
assume that $M$ does not contain $U_{2,4}$, $F_7$ and $M(K_5)$ as a minor.
We also suppose $M$ minimal (for the minor relation) with this property.\\

We first notice that $M$ must be binary (since it contains no $U_{2,4}$-minor).
Moreover $M$ is $2$-connected otherwise, by \eqref{sey1},
$M$ can be written as $M_1 \oplus_1 M_2$, where $M_1$ and $M_2$ are two matroids, but then
by Lemma \ref{lem:elemsum}, one of $M_1,M_2$ (say $M_1$) is such that every
element belongs to at least $3$ triangles, and since both $M_1$ and $M_2$ are
proper minors of $M$ by \eqref{sey5}, then $M_1$ contradicts the minimality of $M$. Now suppose
that $M$ is $2$-connected but not $3$-connected, so by \eqref{sey2}, $M$
can be written as a $2$-sum of $M_1$ and $M_2$ and since $M$ is such that
each element belongs to at least $3$ triangles, by Lemma \ref{lem:elemsum},
each element of $si(M_1)$ except the ones of $E(M_1) \cap E(M_2)$ belongs
to at least $3$ triangles. But since $|E(M_1) \cap E(M_2)| \leq 1$ (by
definition of $2$-sum) and $si(M_1)$ is a proper minor of $M$,
then $si(M_1)$ contradicts the minimality of $M$. So we can assume that
$M$ is $3$-connected.\\

Since $M$ is binary and without $F_7$-minor then, by \eqref{sey}, either $M$ is
isomorphic to $F_7^*$, either $M$ is a regular matroid or $M$ can be written as $2$-sum of
two smaller matroids. But since $M$ is $3$-connected, by \eqref{sey5}, the
latter does not hold and for the former, it is easy to check that no element of
$F_7^*$ belongs to at least three triangles, a contradiction. So $M$ is a
$3$-connected regular matroid.\\

By Seymour's regular matroid characterization \cite{Seymour1},
$M$ is either graphic, cographic, isomorphic to $R_{10}$ or is a $3$-sum
of smaller matroids.\\

Suppose that $M$ is isomorphic to $R_{10}$. Note that for every element $e\in E(R_{10})$,
we have that $R_{10}\setminus e$ is isomorphic to $M(K_{3,3})$. Since $M(K_{3,3})$
is triangle free then every element of $R_{10}$ should be contained in every
triangle of $R_{10}$ implying that every triangle contains 10 elements, which
is a contradiction. Thus $R_{10}$ is triangle-free, a contradiction.
Moreover by Lemmas \ref{lem1} and \ref{lem2}, $M$ is neither graphic nor cographic.
Thus, $M$ can be written as a $3$-sum of smaller matroids.
Suppose that $M = M_1 \oplus_3 M_2$. Since the  only elements of $M$ not belonging to three triangles
of $M$ are either a single element or elements that belongs to a triangle of $M$,
then these elements are contained either in $M_1$ or $M_2$.
Without loss of generality we can assume that they are contained in $M_2$.
But then since $M$ is $3$-connected and binary then, by \eqref{sey5}, $si(M_1)$ is a proper minor
of $M$  and, by Lemma \ref{lem:elemsum}, is such that every element except maybe the elements
of $E(si(M_1)) \cap E(M_2)$ belong to at least $3$ triangles. This contradicts the minimality of $M$. \qed

\section{Proof of Theorem~\ref{th:ternary}}

In this section, we will prove Theorem~\ref{th:ternary} using the following theorem of Oxley \cite{Oxley1}.
\begin{align}\label{th:oxleytern}
&\text{Any $3$-connected ternary matroid with no $M(K_4)$ minor is either isomorphic} \\
&\text{to a whirl $W^r$, to the matroid $J$ or to one of the $15$ $3$-connected minors of}\nonumber \\
&\text{the Steiner matroid $S(5,6,12)$.}\nonumber
\end{align}

We will first prove the following lemma about $3$-connected matroids.
\begin{lemma}\label{lem:3connectter}
Let $M$ be a $3$-connected ternary matroid with no $M(K_4)$-minor with at least $2$ elements such that
every element belongs to at least $2$ triangles, except maybe for one element $e$, then $M$ contains
$P_7$ as a minor or is isomorphic to $U_{2,4}$.
\end{lemma}

\begin{proof}
By the (\ref{th:oxleytern}), $M$ is isomorphic to a whirl $W^r$, to $J$ or is isomorphic to a
$3$-connected minor of the Steiner matroid $S(5,6,12)$. Every whirl $W^r$ for $r \geq 3$ has at least $2$ éléments that does not belongs
to at least two triangles and the matroid $J$ has a $P_7$ minor (\cite[(2,9)]{Oxley1}). Moreover we checked by computer that all the $3$-connected minors
of $S(5,6,12)$ has at least $2$ elements that does not belongs to at least two triangles or contain a $P_7$ minor. So either $M$ contain $P_7$ as a minor
or $M$ is isomorphic to the whirl $W^2$, that is, $M$ is isomorphic to $U_{2,4}$, and the result follows.
\end{proof}

We may now prove Theorem~\ref{th:ternary}.

\begin{proof}[Proof of Theorem~\ref{th:ternary}]
Let $M$ be a simple ternary matroid with no $M(K_4)$ minor such that every elements belongs to at least $2$ triangles.
If $M$ is $3$-connected then, by Lemma~\ref{lem:3connectter}, the result follows.

Suppose now that $M$ is not $3$-connected. By (\ref{sey1}) and (\ref{sey5}), $M$ can be written as $M_1 \oplus_k M_2$ where $k \leq 2$ and where
$M_1$ and $M_2$ are two strict minors of $M$. Without loss of generality, we can suppose that $M_1$ is $3$-connected (by taking $M_1$ and $M_2$
such that $|E(M_1)|$ is minimal). Moreover, by Lemma~\ref{lem:elemsum}, every element of $M_1$ belongs to at least $2$ triangles except maybe
for the only element of $E(M_1) \cap E(M_2)$. So by the (\ref{th:oxleytern}), $M_1$ contains $P_7$ as a minor or is isomorphic to $U_{2,4}$.
In the first case, since $M_1$ is a minor of $M$, then $M$ contain $P_7$ as a minor and we are done. In the second case, suppose by contradiction that $M$
does not contain $U_{2,4}$ as a submatroid. If $M$ is the direct sum of $M_1$ and $M_2$, then $M_1$ is a submatroid of $M$ and thus $M$ contain
$U_{2,4}$ as a submatroid, contradicting the hypothesis. We thus deduce that $M$ is the $2$-sum of $M_1$ and $M_2$. Let $p$ be the only element of $E(M_1) \cap E(M_2)$.
We claim that every element of $E(M_1) \setminus \{p\}$ belongs to at most one triangle in $M$. Suppose that one element of $M_1 \setminus \{p\}$
belongs to two triangles. As $|E(M_1) \setminus \{p\}| = 3$, one of the two triangles denoted by $T$, can be written, by the definition of $2$-sum, as $C_1 \Delta C_2$
where $C_i$ is a circuit of $M_i$ for $1 \leq i \leq 2$. Since $|T| = |C_1| + |C_2| - 2|C_1 \cap C_2| = 3$ and $|C_1 \cap C_2| \leq 1$, we deduce
that either $|C_1| \leq 3$ and $|C_2| = 2$, either $|C_1| = 2$ and $|C_2| \leq 3$. The latter cannot happen because otherwise $C_1$ would be a circuit
of $M_1$ of size $2$ which is not possible since $M_1$ is isomorphic to $U_{2,4}$. In the former case, since $|C_2| = 2$ and $p \in C_2$
(by definition of the $2$-sum), we may denote $C_2 = \{p,q\}$. Since $p \in M_2$ and $q$ is parallel to $p$, $M_{|E(M_1) \setminus \{q\}}$ is isomorphic
to $U_{2,4}$ and thus $M$ contain $U_{2,4}$ as a submatroid, which is again a contradiction. Thus every element of $E(M_1) \setminus \{p\}$
belong to at most one triangle in $M$. Therefore all elements of $M|_{E(M_1) \setminus \{p\}}$ belong to at most one triangle, contradicting the hypothesis, and
the result follows.
\end{proof}

\end{document}